\documentclass[12pt, a4paper]{amsart}

\usepackage[hmargin=30mm, vmargin=25mm, includefoot, twoside]{geometry}
\usepackage[bookmarksopen=true]{hyperref}

\usepackage{amscd}
\usepackage{amsfonts,amssymb,verbatim}
\usepackage{latexsym}
\usepackage{mathrsfs}
\usepackage{stmaryrd}
\usepackage{xspace}
\usepackage{enumerate, paralist}
\usepackage{graphicx}
\usepackage[all]{xy}
\usepackage{extarrows}

\usepackage[usenames,dvipsnames]{color}

\usepackage{txfonts, pxfonts}

\usepackage{amsthm}
\usepackage{amsmath}

\usepackage{todonotes}
\newtheorem{thm}{Theorem}[section]
 \newtheorem{cor}[thm]{Corollary}
 \newtheorem{lem}[thm]{Lemma}
 \newtheorem{prop}[thm]{Proposition}

\numberwithin{equation}{section}

 \theoremstyle{definition}
  \newtheorem{defn}[thm]{Definition}
  \newtheorem{question}[thm]{Question}

 \theoremstyle{remark}
 \newtheorem{rem}[thm]{Remark}
  \newtheorem{ex}[thm]{Example}

\newtheorem*{claim*}{Claim}

\def\NN{\mathbb{N}}

\def\CC{\mathbb{C}}

\def\Nd{\mathcal{N}}

\def\supp{\mathrm{supp}}

\def\diam{\mathrm{diam}}

\begin{document}

\title{Characterisations for uniform amenability}

\author{Jingming Zhu and Jiawen Zhang}

\address[J. Zhu]{College of Data Science, Jiaxing University, Jiaxing , 314001, P.R.China.}
\email{jingmingzhu@zjxu.edu.cn}

\address[J. Zhang]{School of Mathematical Sciences, Fudan University, 220 Handan Road, Shanghai, 200433, China.}
\email{jiawenzhang@fudan.edu.cn}

\date{}

\thanks{The first author is supported by National Natural Science Foundation of China under Grant No.12071183. The second author is supported by National Natural Science Foundation of China under Grant No.11871342.}

\begin{abstract}
In this paper, we provide several characterisations for uniform amenability concerning a family of finitely generated groups. More precisely, we show that the Hulanicki-Reiter condition for uniform amenability can be weakened in several directions, including cardinalities of supports and certain operator norms.
%Let $\{G_n\}_n$ be a sequence of finite groups and let $X$ be the coarsely disjoint union of $\{G_n\}_n$, i.e., $X=$ as $\bigcup_{n}G_n$ which has bounded geometry. The main result of this paper is that $X$ has Yu's Property A if and only if the kernels $K_n$ in $B(l^1(G_n))$ defined either by the Hulanicki-Reiter function or Higson-Roe functions on each $G_n$ are uniformly bounded.
\end{abstract}

\date{\today}
\maketitle

\parskip 4pt

%\textit{Mathematics Subject Classification} (2020): 20F65, 20F69, 51F30, 43A07.


\section{Introduction}\label{sec:intro}

The notion of amenability for groups was introduced by von Neumann \cite{vN29} as the obstruction to the Banach–Tarski paradox \cite{BT24}. Amenability admits a large number of characterisations, and is now fundamental in many areas of mathematics. There are a lot of excellent references on amenability and we guide interested readers to, \emph{e.g.}, \cite{Pat88, Run02, Wag85}.

%Recall from the Hulanicki-Reiter condition (\cite{Hul66, Rei65}) that a finitely generated group $G$ with finite generating set $S$ is called \emph{amenable} if for any $\varepsilon>0$, there exists a finitely supported function $f: G \to [0,1]$ such that $\sum_{\gamma\in G} |f(\gamma)|^2 =1$ and $\sum_{\gamma \in G}|f(\gamma) - f(s^{-1}\gamma)|^2 \leq \varepsilon^2$ for any $s\in S$.

On the other hand, Yu introduced the notion of Property A for discrete metric spaces \cite{Yu00} to attack the coarse Baum-Connes conjecture, and hence the Novikov conjecture in the case of groups. Property A can be regarded as a non-equivariant version of amenability, and has gained a lot of attractions due to its close relations with operator algebras and other areas of mathematics (see, \emph{e.g.}, \cite{Wil09}).

In this paper, we mainly focus on the case of a family of groups and study a uniform version of amenability and Property A for the family. Roughly speaking, a family of groups is uniformly amenable if the quantifiers in the definition of amenability can be made uniform. More precisely, we consider the following:

\begin{defn}\label{defn: uniformly amenable}
Let $\{G_i\}_{i\in I}$ be a family of finitely generated groups with finite generating set $S_i$ for $i\in I$. We say that $\{(G_i,S_i)\}_{i\in I}$ is \emph{uniformly amenable} if for any $\varepsilon>0$ there exists $D>0$ and a family of functions $\{f_i: G_i \to [0,1]\}_{i\in I}$ such that $\sum_{\gamma\in G_i} |f_i(\gamma)|^2 =1$, $\sum_{\gamma \in G_i} |f_i(\gamma) - f_i(\gamma s)|^2 \leq \varepsilon^2$ for any $s\in S_i$ and $\supp(f_i) \subseteq B(1_{G_i}, D)$.
\end{defn}

The main contribution of this paper is to provide several characterisations for uniform amenability. To be more precise, we first consider the general case and obtain two conditions equivalent to uniform amenability (see Theorem \ref{thm:main result general case}), which indicate that the uniform boundedness on the diameter of $\supp(f_i)$ in Definition \ref{defn: uniformly amenable} can be relaxed to the uniform boundedness on their cardinalities, or even to that on the $\ell^1$-norm of $f_i$. Furthermore, we realise that when the group $G_i$ is amenable, the square of the $\ell^1$-norm of $f_i$ coincides with the operator norm of certain linear bounded operator on $\ell^2(G_i)$ associated to the function $f_i$ (see Proposition \ref{prop:op norm and l1 norm}). Therefore, in this special case we obtain two extra characterisations for uniform amenability, one of which also relates to uniform Property A (see Theorem \ref{thm:main result amenability}). Applying our main results to a sequence of finite groups, we obtain characterisations for Property A on their coarse disjoint union (see Corollary \ref{cor:main result amenability}).

The paper is organised as follows: After recalling the background notions and notation in Section \ref{sec:pre}, we state and prove our main results in Section \ref{sec:main results}. Finally, we give some couterexamples and open questions in Section \ref{sec:ques}.

\subsection*{Acknowledgements}
The first-named author would like to thank Prof. Bolin Ma for helpful discussions. The second-named author would like to thank Prof. Yijun Yao for helpful discussions. We would also like to thank the anonymous referee for pointing out a gap in the previous version.

%Our focus in this paper is on property A, which was introduced by G. Yu in \cite{Yu00} as a weak form of amenability. He proved that a property A metric space $X$ coarsely embeds into a Hilbert space, that such coarse embeddability implies coarse Baum-Connes conjecture. For the metric spaces with bounded geometry (for precise definitions see Section 2), Higson and Roe introduced an equivalent definition of Property A in an analytic way. Thinking of operators on $\ell^2(X)$ as $X$-by-$X$ matrices, we say that such an operator has finite propagation (or is a band operator), if the non-zero entries appear only in a band of finite width (measured by the metric on $X$) around the main diagonal. For a metric space $X$, one of the equivalent condition that $X$ has property A is that the kernel $K$ defined by
%$$K(x,y)=\langle \xi_x,\xi_y \rangle\text{ for any }x,y\in X$$
%is a bounded operator on $\ell^2(X)$ with finite propagation \cite{Willett}, in which $\xi_x,\xi_y\in \ell^2(X)_{1,+}$ are the Higson-Roe functions in the analytic definition of Property A. In this paper, we generalize this result when $X$ is the disjoint union of a sequence of finite groups.
%
%The paper is structured as follows: After recalling the background notations, terminology and some existing results and examples in Section 2, we state and discuss the main result in Section 3. The paper closes with applications and open questions in Section 4.

\section{Preliminaries}\label{sec:pre}

Here we collect some necessary notions for this paper. Readers are guided to \cite{NY12,Wil09} for general background and more details.

\subsection{Notions from coarse geometry}
Let us start with some basic notions.

\begin{defn}\label{defn: basic notions from coarse geometry}
Let $(X, d)$ be a discrete metric space.
\begin{enumerate}
 \item For $x\in X$ and $R>0$, denote the closed ball $B(x,R):=\{y\in X: d(x,y) \leq R\}$.
 \item We say that $(X,d)$ has \emph{bounded geometry} if $\sup_{x\in X}\sharp B(x, R)$ is finite for any $R>0$. Here we use the notation $\sharp A$ to denote the cardinality of the subset $A \subseteq X$.
 \item For a subset $A \subseteq X$ and $R>0$, denote $\Nd_R (A):=\{x\in X:d(x,A)\leq R\}$ the $R$-neighborhood of $A$ in $X$. Also denote $\partial_R A:=\{x\in X\setminus A : d(x,A)\leq R\}$ the $R$-boundary of $A$ in $X$.
 \item For $R>0$, we say that a subset $A \subseteq X$ is \emph{$R$-connected} if for any $x,y\in A$ there exists a sequence $x_0=x, x_1, x_2, \cdots, x_n=y$ in $A$ such that $d(x_i,x_{i+1}) \leq R$ for $i=0,1,\cdots, n-1$.
 \item For a subset $A \subseteq X$, we denote $\chi_A$ the characteristic function of $A$. For $x\in X$, also denote $\delta_x:=\chi_{\{x\}}$.
\end{enumerate}
\end{defn}

\begin{defn}\label{defn: coarse equivalence}
Let $(X,d_X)$ and $(Y,d_Y)$ be metric spaces and $f: X \to Y$ be a map. We say that $f$ is a \emph{coarse embedding} if there exist functions $\rho_{\pm}: [0,\infty) \to [0,\infty)$ with $\lim_{t\to +\infty}\rho_{\pm}(t) = +\infty$ such that for any $x,y\in X$ we have
\[
\rho_-(d_X(x,y)) \leq d_Y(f(x),f(y)) \leq \rho_+(d_X(x,y)).
\]
If additionally there exists $C>0$ such that $Y=\Nd_C(f(X))$, then we say that $f$ is a \emph{coarse equivalence} and $(X,d_X), (Y,d_Y)$ are \emph{coarsely equivalent}.
\end{defn}

Discrete groups provide a large class of interesting examples of metric spaces with bounded geometry (see, \emph{e.g.}, \cite[Chapter 1]{NY12}). For a finitely generated group $G$ with a finite generating set $S$, there is the \emph{word length function $\ell_S$} defined by
\[
\ell_S(g) = \min\{k: g=s_{i_1}^{\pm}s_{i_2}^{\pm}\cdots s_{i_k}^{\pm}\mbox{~where~} s_{i_j} \in S \mbox{~for~}1 \leq j \leq k\}.
\]
Note that $\ell_S$ is a proper function on $G$ in the sense that $\sharp \{g: \ell_S(g) \leq R\}$ is finite for all $R\geq 0$. This leads to a discrete left-invariant metric $d_S$ on $G$ by $d_S(g,h):=\ell_S(g^{-1}h)$, which is clearly of bounded geometry. We call $d_S$ the \emph{word length metric} on $G$ (associated to $S$). Gromov \cite{Gro93} observed that different finite generating sets $S$ give coarsely equivalent word length metrics on $G$ (see also \cite[Theorem 1.3.12]{NY12}).

%In the sequel, we also need to consider the case of families of metric spaces.
%
%\begin{defn}\label{defn:}
%A family of metric spaces $\{(X_i,d_i)\}_{i\in I}$ is said to have \emph{uniformly bounded geometry} if $\sup_{i\in I}\sup_{x\in X_i}\sharp B(x, R)$ is finite for any $R>0$.
%\end{defn}

%Our main focus will be a single metric space built out of a sequence of finite metric spaces.

Given a sequence of finite metric spaces, we can build a single metric space encoding the information of coarse geometry of the sequence in a uniform way:

\begin{defn}\label{defn: coarse disjoint union}
Let $\{(X_n,d_n)\}_{n\in \NN}$ be a sequence of finite metric spaces. \emph{A coarse disjoint union} of $\{X_n\}_n$ is a metric space $(X,d)$ such that $X=\bigsqcup_n X_n$ as a set and the metric $d$ coincides with $d_n$ on each $X_n$ and satisfies:
\[
d(X_n, X_m) \geq n+m+\diam(X_n) + \diam(X_m) \quad \mbox{for~each~}n \neq m \mbox{~in~} \NN.
\]
\end{defn}

Note that the metric $d$ in Definition \ref{defn: coarse disjoint union} is \emph{not} uniquely determined, however, it is easy to see that they are mutually coarsely equivalent to each other.
%Also it follows directly from definitions that a coarse disjoint union of finite metric spaces $\{X_n\}_n$ has bounded geometry \emph{if and only if} the sequence $\{X_n\}_n$ has uniformly bounded geometry.

Also recall the following notion:
\begin{defn}\label{defn: variation}
Let $(X,d_X)$ and $(Y,d_Y)$ be metric spaces and $f: X \to Y$ be a map. Given $\varepsilon, R>0$, we say that the map $f$ has \emph{$(\varepsilon,R)$-variation} if for any $x,y\in X$ with $d_X(x,y) \leq R$, then $d_Y(f(x),f(y)) \leq \varepsilon$.
\end{defn}

\subsection{Amenability}\label{ssec:amenability}

%Here we briefly recall two key notions in group theory and coarse geometry, which will be the main focus of this paper. Readers are guided to \cite[Chapter 3 and 4]{NY12} for more details.

The notion of amenability was originally introduced by von Neumann \cite{vN29} as the obstruction to the Banach-Tarski paradox. Here we use the following equivalent definition due to F{\o}lner \cite{Fol55, Fol57}:

\begin{defn}\label{defn: amenability}
Let $G$ be a finitely generated group equipped with a word length metric. We say that $G$ is \emph{amenable} if for any $R>0$ and $\varepsilon>0$, there exists a finite subset $F\subseteq G$ such that
\[
\frac{\sharp \partial_R F}{\sharp F}\leq \varepsilon.
\]
\end{defn}

We also need an analytic characterisation for amenability. For a set $X$ and $p\in[1,\infty)$, we denote
\[
\ell^p(X):=\big\{f: X \to \CC ~\big|~ \sum_{x\in X}|f(x)|^p < \infty\big\}.
\]
For $f\in \ell^p(X)$, we endow the $\ell^p$-norm by $\|f\|_p:=(\sum_{x\in X}|f(x)|^p)^{1/p}$. Also denote
\[
\ell^p(X)_{1,+}=\big\{f\in l^p(X): \|f\|_p=1\text{ and }f(x)\geq 0 \text{ for any }x\in X\big\}.
\]
Given a group $G$, a function $f: G \to \CC$ and $\gamma \in G$, denote the \emph{right $\gamma$-translation} of $f$ by $\gamma \cdot f$. In other words, $(\gamma \cdot f)(h)=f(h\gamma)$ for any $h\in G$. We also need to consider the \emph{left $\gamma$-translation} of $f$ defined by $(\gamma \star f)(h):=f(\gamma^{-1}h )$ for $h\in G$. 

%Hulanicki and Reiter \cite{Hul66, Rei65} considered the following condition for $p=1$:
%
%\begin{defn}\label{defn: Hul-Rei condition}
%Let $G$ be a finitely generated group equipped with a word length metric and $p \in [1,\infty)$. Given $R,\varepsilon>0$, a function $f \in \ell^p(G)_{1,+}$ is called an \emph{$\ell^p$-Hulanicki-Reiter function with $(\varepsilon,R)$-variation} if $f$ has $(\varepsilon,R)$-variation and $\supp(f)$ is finite.
%\end{defn}

The following result was essentially due to Hulanicki and Reiter\footnote{They originally consider the $\ell^1$-case, while this is easily seen to be equivalent to the $\ell^2$-case shown in Proposition \ref{prop: Hul-Rei condition} using the Mazur maps (see, \emph{e.g.}, \cite[Definition (Theorem) 6.1.1]{Wil09}).} \cite{Hul66, Rei65}: %(see also \cite[Theorem 3.2.2]{NY12}):

\begin{prop}[Hulanicki-Reiter condition]\label{prop: Hul-Rei condition}
Let $G$ be a finitely generated group with finite generating set $S$ and the associated word length metric. Then the following are equivalent:
\begin{enumerate}
 \item $G$ is amenable;
  \item For any $\varepsilon>0$, there exists a function $f\in \ell^2(G)_{1,+}$ such that $\|f-s\cdot f\|_2\leq \varepsilon$ for any $s\in S$ and $\supp(f)$ is finite.
\end{enumerate}
\end{prop}

We remark that since word length metrics on groups are indeed graph metrics on their Cayley graphs, condition (2) above is also equivalent to the following: for any $\varepsilon>0$ and $R>0$, there exists a function $f\in \ell^2(G)_{1,+}$ such that $\|f-\gamma\cdot f\|_2\leq \varepsilon$ for any $\gamma \in G$ with $\ell_S(\gamma) \leq R$ and $\supp(f)$ is finite. This form resembles the original Hulanicki-Reiter condition for the $\ell^1$-case.

As mentioned in Section \ref{sec:intro}, our main focus will be the notion of uniform amenability for a family of groups  (see Definition \ref{defn: uniformly amenable}). Roughly speaking, it requires that condition (2) in Proposition \ref{prop: Hul-Rei condition} holds uniformly for the given family. In fact, this is also equivalent to a uniform version of the F{\o}lner condition (Definition \ref{defn: amenability}):

\begin{prop}\label{prop: uniform Hul-Rei condition}
Let $\{G_i\}_{i\in I}$ be a family of finitely generated groups with finite generating set $S_i$ such that $\sup_i \sharp S_i<\infty$ and the associated word length metric. Then the following are equivalent:
\begin{enumerate}
 \item $\{(G_i, S_i)\}_{i\in I}$ is uniformly amenable, \emph{i.e.}, for any $\varepsilon>0$ there exists $D>0$ and a family of functions $\{f_i\}_{i\in I}$ such that each $f_i \in \ell^2(G_i)_{1,+}$ satisfies $\|f_i-s\cdot f_i\|_2\leq \varepsilon$ for $s\in S_i$ and $\supp(f_i) \subseteq B(1_{G_i}, D)$.
 \item For any $R>0$ and $\varepsilon>0$, there exists $D'>0$ and a family of finite subsets $\{F_i\subseteq G_i\}_{i\in I}$ such that $\sharp \partial_R F_i/\sharp F_i \leq \varepsilon$ and $F_i \subseteq B(1_{G_i}, D')$.
\end{enumerate}
\end{prop}

The proof follows directly from that of Proposition \ref{prop: Hul-Rei condition} (see, \emph{e.g.}, \cite[Theorem 3.2.2]{NY12}) while noticing that the quantifiers can be uniformly controlled thanks to the assumption of $\sup_i {\sharp S_i}<\infty$, hence we omit the details. It is worth noticing that the condition $\sup_i {\sharp S_i}<\infty$ is also equivalent to that the coarse disjoint union $\bigsqcup G_i$ has bounded geometry.

%\begin{defn}\label{defn: uniformly amenable}
%Let $\{G_i\}_{i\in I}$ be a family of discrete groups with finite generating set $S_i$. We say that $\{(G_i,S_i)\}_{i\in I}$ is \emph{uniformly amenable} if for any $\varepsilon>0$ there exists $D>0$ and a family of functions $\{f_i\}_{i\in I}$ such that each $f_i \in \ell^2(G_i)_{1,+}$ satisfies $\|f_i-s\cdot f_i\|_2\leq \varepsilon$ for $s\in S_i$ and $\supp(f_i) \subseteq B(1_{G_i}, D)$.
%\end{defn}

We also remark that changing the generating sets $S_i$ might change the coarse geometry of the family $\{G_i\}_{i\in I}$, since the control function $\rho_{\pm}$ in Definition \ref{defn: coarse equivalence} might alter according to $i \in I$. Hence when talking about uniform amenability, we have to fix a generating set for each group in the family.

\subsection{Property A}
The notion of Property A was introduced by Yu in his celebrated work \cite{Yu00} to attack the coarse Baum-Connes conjecture. Intuitively speaking, Property A is a non-equivariant version of amenability, hence these two notions enjoy analogous descriptions.

\begin{defn}\label{defn: Property A}
A discrete metric space $(X,d)$ is said to have property A if for any $\varepsilon, R>0$ there exists a collection of finite subsets $\{A_x\}_{x\in X}$ in $X\times\NN$ and a constant $S>0$ such that
\begin{enumerate}
  \item $\frac{\sharp(A_x\Delta A_y)}{\sharp(A_x\cap A_y)}\leq \varepsilon$ for $x,y\in X$ with $d(x,y)\leq R$;\\[0.2pt]
  \item $A_x\subset B(x,S)\times \mathbb{N}$ for any $x\in X$.
\end{enumerate}
\end{defn}

Analogous to the case of amenability, we also need the following analytic characterisation for Property A due to Higson and Roe\footnote{Again note that the $\ell^2$-case shown in Proposition \ref{prop: Hig-Roe condition} is equivalent to the $\ell^1$-case using the Mazur maps (see, \emph{e.g.}, \cite[Theorem 1.2.4]{Wil09}).} \cite{HR00}.

%\begin{defn}\label{defn: Hig-Roe condition}
%Let $(X,d)$ be a discrete metric space and $p \in [1,\infty)$. Given $R,\varepsilon>0$, we say that a map $\xi: X\to \ell^p(X)_{1,+}$ (we will use $\xi_x$ to denote $\xi(x)$ for $x\in X$) is an \emph{$\ell^p$-Higson-Roe function with $(\varepsilon,R)$-variation} if $\xi$ has $(\varepsilon,R)$-variation and there exists a constant $S>0$ such that $\supp(\xi_x) \subset B(x,S)$ for any $x\in X$.
%\end{defn}

%
%The following result was essentially proved by Higson and Roe \cite{HR00} in the case of $p=1$ (see also \cite[Theorem 4.2.1]{NY12} and \cite[Theorem 1.2.4]{Wil09}):

\begin{prop}[Higson-Roe condition]\label{prop: Hig-Roe condition}
Let $(X,d)$ be a discrete metric space of bounded geometry. Then the following are equivalent:
\begin{enumerate}
 \item $(X,d)$ has Property A;
 \item For any $\varepsilon>0$ and $R>0$, there exists a map $\xi: X\to \ell^2(X)_{1,+}$ such that $\xi$ has $(\varepsilon,R)$-variation and there exists a constant $S>0$ such that $\supp(\xi(x)) \subset B(x,S)$ for any $x\in X$.
\end{enumerate}
\end{prop}

\begin{rem}\label{rem:Property A explanation}
Note that when $X$ comes from a finitely generated group equipped with a word length metric, then condition (2) above is also equivalent to the following: for any $\varepsilon>0$, there exists a map $\xi: X\to \ell^2(X)_{1,+}$ such that $\xi$ has $(\varepsilon,1)$-variation and there exists a constant $S>0$ such that $\supp(\xi(x)) \subset B(x,S)$ for any $x\in X$.
\end{rem}

It is well-known that for a finitely generated group $G$ with a word length metric, amenability implies Property A. This can be easily seen either from the definitions or the analytic characterisations. 
%More precisely, for any $\varepsilon>0$ assume that $f \in \ell^2(G)_{1,+}$ satisfies condition (2) in Proposition \ref{prop: Hul-Rei condition}. Then the function $\xi: G \to \ell^2(G)_{1,+}$ defined by $\xi(\gamma):= \gamma \star f$ satisfies condition (2) in Proposition \ref{prop: Hig-Roe condition}.

Similarly, we would also like to consider the following notion of Property A for the family case.

\begin{defn}\label{defn: uniformly property A}
Let $\{X_i\}_{i\in I}$ be a family of discrete metric spaces. We say that $\{X_i\}_{i\in I}$ has \emph{uniform Property A} if for any $\varepsilon>0$ and $R>0$, there exists $D>0$ and a family of functions $\{\xi_i: X_i \to \ell^2(X_i)_{1,+}\}_{i\in I}$ such that each $\xi_i$ has $(\varepsilon,R)$-variation and $\supp(\xi_i(x)) \subset B(x,D)$ for any $x\in X_i$ and $i \in I$.
\end{defn}

\section{Main results}\label{sec:main results}

In this section, we present several versions of the main results under different hypotheses and provide detailed proofs. Let us start with the general case.

\begin{thm}\label{thm:main result general case}
Let $\{G_i\}_{i\in I}$ be a family of finitely generated groups with finite generating set $S_i$ for each $i$. Consider the following conditions:
\begin{enumerate}
 \item For any $\varepsilon>0$, there exists $M>0$ and a family of functions $\{f_i\}_{i\in I}$ such that each $f_i \in \ell^2(G_i)_{1,+} \cap \ell^1(G_i)$ satisfies $\|f_i-s\cdot f_i\|_2\leq \varepsilon$ for $s\in S_i$ and $\|f_i\|_1 \leq M$.
 \item For any $\varepsilon>0$, there exists $N>0$ and a family of functions $\{f_i\}_{i\in I}$ such that each $f_i \in \ell^2(G_i)_{1,+}$ satisfies $\|f_i-s\cdot f_i\|_2\leq \varepsilon$ for $s\in S_i$ and $\sharp \supp(f_i) \leq N$.
 \item The family $\{(G_i,S_i)\}_{i\in I}$ is uniformly amenable.
 %, \emph{i.e.}, for any $\varepsilon>0$ there exists $D>0$ and a family of functions $\{f_i\}_{i\in I}$ such that each $f_i \in \ell^2(G_i)_{1,+}$ has $(\varepsilon,1)$-variation and $\supp(f_i) \subseteq B(1_{G_i}, D)$.
\end{enumerate}
Then (1) $\Rightarrow$ (2) $\Rightarrow$ (3). Additionally, if there exists $L>0$ such that $\sharp S_i \leq L$ for each $i$, then the three conditions above are equivalent.
\end{thm}

Note that when the index set $I$ is finite, Theorem \ref{thm:main result general case} holds trivially by an elementary approximation argument. Hence the main contribution of this result is on the uniformality of parameters.

To prove Theorem \ref{thm:main result general case}, we need the following auxiliary lemma:

\begin{lem}\label{lem:existence of cut}
Let $X$ be a set and $f: X \to [0,1]$ be a function in $\ell^2(X)_{1,+}$ with finite support. Then for any $\varepsilon\in [0,1]$, there exists $c\geq 0$ such that $\|f_c\|_2=\varepsilon$ where $f_c(x):=\min\{f(x), c\}$ for $x\in X$.
\end{lem}

\begin{proof}
It is clear that the function $F(c):=\|f_c\|_2$ for $c\in [0,1]$ is continuous since $\supp(f)$ is finite. Note that $F(0)=0$ and $F(1)=1$, hence the result follows from the Intermediate Value Theorem for continuous functions.
\end{proof}

\begin{proof}[Proof of Theorem \ref{thm:main result general case}]
``(1) $\Rightarrow$ (2)'': Given $\varepsilon>0$, take $M>0$ and a family of functions $\{f_i\}_{i\in I}$ satisfying condition (1) for $\varepsilon$ and $M$. Approximating each $f_i$ by finitely support functions, we can assume that each $f_i$ has finite support. Set $\hat{\varepsilon}:=\frac{\varepsilon}{2+2\varepsilon}$. By Lemma \ref{lem:existence of cut}, there exists $c_i>0$ for each $i \in I$ such that $\|g_i\|_2 = \hat{\varepsilon}$, where $g_i(\gamma):=\min\{f_i(\gamma), c_i\}$ for $\gamma \in G_i$. Set $\supp_{\hat{\varepsilon}}(f_i):=\{\gamma \in G_i: f_i(\gamma) \geq c_i\}$.

We claim: $\sharp \supp_{\hat{\varepsilon}}(f_i)$ is uniformly bounded with respect to $i \in I$. In fact, we have:
\[
\hat{\varepsilon}^2 = \sum_{\gamma \in G_i} g_i(\gamma)^2 \leq c_i \cdot \sum_{\gamma \in G_i} g_i(\gamma)
\]
since $g_i(\gamma) \leq c_i$ for $\gamma \in G_i$. Then we have
\[
M \geq \|f_i\|_1 = \sum_{\gamma \in G_i} f_i(\gamma) \geq \sum_{\gamma \in G_i} g_i(\gamma) \geq \frac{\hat{\varepsilon}^2}{c_i},
\]
which implies that $c_i \geq \hat{\varepsilon}^2/M$ for all $i \in I$. On the other hand, we have
\[
\hat{\varepsilon}^2 = \sum_{\gamma \in G_i} g_i(\gamma)^2 \geq c_i^2 \cdot \sharp \supp_{\hat{\varepsilon}}(f_i) \geq \frac{\hat{\varepsilon}^4}{M^2} \cdot \sharp \supp_{\hat{\varepsilon}}(f_i).
\]
Hence we obtain that $\sharp \supp_{\hat{\varepsilon}}(f_i) \leq M^2/\hat{\varepsilon}^2$, which conclude the claim.

Back to the proof, let us define $h_i:=\frac{f_i-g_i}{\|f_i - g_i\|_2}$ for $i \in I$. Then $\supp(h_i)\subseteq\supp_{\hat{\varepsilon}}(f_i)$, which has uniformly bounded cardinality by the claim above. Also for any $s \in S_i$, we have
\[
\|h_i - s \cdot h_i\|_2 \leq \frac{\|f_i - s \cdot f_i\|_2 + \|g_i - s \cdot g_i\|_2}{\|f_i - g_i\|_2} \leq \frac{\varepsilon + 2\hat{\varepsilon}}{1-\hat{\varepsilon}} = 2\varepsilon.
\]
Hence we conclude condition (2).

``(2) $\Rightarrow$ (3)'': Given $\varepsilon>0$, take $N>0$ and a family of functions $\{f_i\}_{i\in I}$ satisfying condition (2) for $\varepsilon$ and $N$. For each $i \in I$, let $\{U_{i,k}\}_{k=1}^{N_i}$ be the family of $2$-connected components of $\supp(f_i)$, and denote $f_i^{(k)}:=f_i \cdot \chi_{U_k}$ for each $k$. Note that the family $\{U_{i,k}\}_{k=1}^{N_i}$ is mutually at least $3$-apart from each other, hence for any $s \in S_i$ we have:
\[
\varepsilon^2 \geq \|f_i - s \cdot f_i\|_2^2 = \sum_{k=1}^{N_i} \|f_i^{(k)} - s \cdot f_i^{(k)}\|_2^2.
\]
Also note that $N_i \leq N$ and for each $i\in I$ and $k=1,\cdots,N_i$, there exists $\alpha_{i,k} \in G_i$ such that $U_{i,k} \subseteq B(\alpha_{i,k}, 2N)$. 

If $\diam(G_i) \leq 4(N+1)N$, then we simply take $g_i:= \frac{1}{(\sharp G_i)^{1/2}}\chi_{G_i}$. Otherwise, choose a point $\beta_i \in G_i$ such that $d_i(\beta_i, 1_{G_i}) = 4(N+1)N$ and a finite sequence of elements $1_{G_i}=\alpha'_{i,1}, \alpha'_{i,2}, \cdots, \alpha'_{i,N_i}$ on a geodesic connecting $1_{G_i}$ and $\beta_i$ such that $d_i(\alpha'_{i,k}, \alpha'_{i,l}) = 4(N+1)|k-l|$ for $k,l=1,\cdots, N_i$ (note that we do not require $\alpha'_{i,N_i} = \beta_i$). Set $\gamma_{i,k}:=\alpha'_{i,k} \alpha_{i,k}^{-1}$ for $k=1,\cdots,N_i$. Then we have the following:
\begin{enumerate}[(i)]
% \item $1_{G_i} \in \gamma_{i,1} \cdot U_{i,1}$;
 \item $d(\gamma_{i,k} \cdot U_{i,k} , \gamma_{i,l} \cdot U_{i,l} ) \geq 3$ for any $k \neq l$ in $\{1,2,\cdots,N_i\}$;
 \item the union $\bigcup_{k=1}^{N_i}\gamma_{i,k} \cdot U_{i,k} $ is contained in $B(1_{G_i}, 4(N+2)N)$.
\end{enumerate}
%Now Lemma \ref{lem:moving subsets} allows us to choose $\gamma_k \in G_i$ for $k=1,\cdots, N_i$ such that
%\textcolor[rgb]{1.00,0.00,0.00}{\begin{enumerate}[(i)]
% \item $1_{G_i} \in \gamma_1 \cdot U_1$;
% \item $d(\gamma_k \cdot U_k , \gamma_l \cdot U_l ) \geq 3$ for any $k \neq l$ in $\{1,2,\cdots,N_i\}$;
% \item the union $\bigcup_{k=1}^{N_i}\gamma_k \cdot U_k $ is $3$-connected.
%\end{enumerate}
%}
Set $g_i:= \sum_{k=1}^{N_i} \gamma_{i,k} \star f_i^{(k)}$, where $\gamma_{i,k} \star f_i^{(k)}$ is the left $\gamma_{i,k}$-translation of $f_i^{(k)}$ defined in Section \ref{ssec:amenability}. It follows from the construction (see condition (i) above) that
\[
\|g_i\|^2_2 = \sum_{k=1}^{N_i} \|\gamma_{i,k} \star f_i^{(k)}\|_2^2 = \sum_{k=1}^{N_i} \|f_i^{(k)}\|_2^2 = \|f_i\|^2_2 =1,
\]
and
\begin{align*}
\|g_i - s \cdot g_i\|^2_2 &= \sum_{k=1}^{N_i} \|\gamma_{i,k} \star f_i^{(k)} - s\cdot (\gamma_{i,k} \star f_i^{(k)})\|^2_2 = \sum_{k=1}^{N_i} \|\gamma_{i,k} \star (f_i^{(k)} - s\cdot f_i^{(k)})\|^2_2\\
&= \sum_{k=1}^{N_i} \|f_i^{(k)} - s\cdot f_i^{(k)}\|^2_2 \leq \varepsilon^2
\end{align*}
for $s \in S_i$, where we use the commutativity of the left translation with the right translation in the second equality. Finally, it follows from condition (ii) above that
\[
\supp(g_i) \subseteq \bigcup_{k=1}^{N_i} \gamma_{i,k} \cdot U_{i,k} \subseteq B(1_{G_i}, 4(N+2)N).
\]
Hence we conclude the proof.
%Finally, the equivalence between (3) and (4) follows from \cite[Theorem 4.4.4]{NY12} together with Proposition \ref{prop: Hig-Roe condition}, and the ``additional'' part is straightforward.
\end{proof}

\begin{rem}\label{rem:quan.expln.}
As mentioned above, the main contribution of Theorem \ref{thm:main result general case} is to provide quantitative relations among these conditions. Indeed it follows from the proof that for $\varepsilon, M$ in condition (1), we may choose $2\varepsilon, N=\frac{4M^2(1+\varepsilon)^2}{\varepsilon^2}$ for condition (2). For $\varepsilon', N'$ in condition (2), we may choose $\varepsilon', 4(N'+2)N'$ for condition (3).
\end{rem}

Note that in Theorem \ref{thm:main result general case}, we do not require the assumption of amenability on each $G_i$ at the very beginning, while this becomes an outcome of conditions therein. In the following, we will explore the situation when we already know that each $G_i$ is amenable, \emph{e.g.}, $G_i$ is finite.

Let us first recall the following, which is the $\ell^2$-version of \cite[Theorem 4.4.4]{NY12}.

\begin{lem}\label{lem:thm 4.4.4 in NY12}
Let $G$ be a finitely generated amenable group with finite generating set $S$. Then for any $\varepsilon>0$ and $D>0$, the following are equivalent:
\begin{enumerate}
 \item There exists a function $f \in \ell^2(G)_{1,+}$ such that $\|f-s\cdot f\|_2\leq \varepsilon$ for $s\in S$ and $\supp(f) \subseteq B(1_G, D)$.
 \item There exists a map $\xi: G \to \ell^2(G)_{1,+}$ such that $\xi$ has $(\varepsilon,1)$-variation and $\supp(\xi_x) \subseteq B(x, D)$.
\end{enumerate}
\end{lem}

\begin{proof}
Applying the Mazur maps
\[
M_{2,1}: \ell^2(G)_{1,+} \to \ell^1(G)_{1,+} \mbox{~defined~by~} (M_{2,1}f)(x):=f(x)^2, x\in G
\]
and
\[
M_{1,2}: \ell^1(G)_{1,+} \to \ell^2(G)_{1,+} \mbox{~defined~by~} (M_{1,2}f)(x):=\sqrt{f(x)}, x\in G,
\]
it is easy to see (\emph{e.g.}, \cite[Theorem 1.2.4, (2) $\Rightarrow$ (3)]{Wil09}) that condition (1) and (2) are equivalent to their $\ell^1$-counterparts (\emph{i.e.}, replacing $\ell^2(G)_{1,+}$ by $\ell^1(G)_{1,+}$ therein)  without changing the supports, respectively. Now we conclude the result thanks to \cite[Theorem 4.4.4]{NY12}.
\end{proof}

Note that the emphasis of the above lemma is on the parameter $D$. Hence it helps to bridge the condition of uniform amenability in Theorem \ref{thm:main result general case} with uniform Property A, which will be packaged into Theorem \ref{thm:main result amenability}.

On the other hand, we would also like to explore an alternative description for condition (1) in Theorem \ref{thm:main result general case} in terms of operator norms under the assumption of amenability, where the operator is defined as follows:

\begin{defn}\label{defn:associated kernels}
For a finitely generated group $G$ and a function $f \in \ell^1(G) \cap \ell^2(G) = \ell^1(G)$ with range in $[0,1]$, the \emph{associated kernel of $f$} is the $G$-by-$G$ matrix $K_f$ defined by $K_f(x,y):= \langle x \star f, y \star f \rangle_{\ell^2(G)}$ for $x,y\in G$.
\end{defn}

For any $G$-by-$G$ matrix $K$ with complex values, we can associate an operator $T_K$ on the space $C_c(G)$ of finitely supported complex-valued functions on $G$ to itself by formal multiplication:
\[
\left(T_K(h)\right)(\gamma) :=\sum_{\gamma' \in G} K(\gamma, \gamma')h(\gamma')
\]
for any $h\in C_c(G)$ and $\gamma \in G$. To simplify the notation, we denote $T_f$ instead of $T_{K_f}$ for $f \in \ell^1(G)$.

The following result relates the operator norm of $T_f$ with the $\ell^1$-norm of $f$.

\begin{prop}\label{prop:op norm and l1 norm}
For a finitely generated amenable group $G$ and a function $f \in \ell^1(G)$ with range in $[0,1]$, the operator $T_f$ can be continuously extended to a bounded linear operator on $\ell^2(G)$ with operator norm equals $\|f\|_1^2$.
\end{prop}

\begin{proof}
Given $\alpha=\sum_{\gamma\in G}\alpha(\gamma) \delta_\gamma \in C_c(G)$, we have:
\begin{align*}
\langle &T_f \alpha, \alpha \rangle_{\ell^2(G)} = \sum_{x,y \in G} \alpha(x)\overline{\alpha(y)} K_f(y,x) = \sum_{x,y \in G} \alpha(x)\overline{\alpha(y)} \cdot \langle y \star f, x \star f \rangle_{\ell^2(G)}\\
&=\sum_{x,y,z \in G}\alpha(x)\overline{\alpha(y)}f(x^{-1}z)f(y^{-1}z) = \sum_{z\in G}\left( \sum_{x\in G} \alpha(x)f(x^{-1}z)\right) \cdot \overline{\left( \sum_{y\in G} \alpha(y)f(y^{-1}z)\right)}\\
&=\sum_{z\in G} \left| \sum_{x\in G} \alpha(x)f(x^{-1}z) \right|^2 = \|\alpha \ast f\|_2^2,
\end{align*}
where $\alpha \ast f$ is the standard convolution of functions on groups. Note that $\alpha \in C_c(G)$ and $f \in \ell^1(G)$, hence it follows from Young's inequality that
\[
\|\alpha \ast f\|_2 \leq \|\alpha\|_2 \cdot \|f\|_1.
\]
Therefore using the polarization identity, $T_f$ can be extended to a bounded linear operator on $\ell^2(G)$, still denoted by $T_f$. It is obvious that $T_f$ is symmetric, hence we obtain:
\[
\|T_f\|=\sup\big\{|\langle T_f \alpha, \alpha \rangle|: \alpha \in \ell^2(G) \mbox{~with~} \|\alpha\|_2=1\big\} \leq \|f\|_1^2.
\]

Since $G$ is amenable, we can choose a sequence of F{\o}lner sets $\{F_n\}_{n\in \NN}$, \emph{i.e.}, each $F_n$ is a finite subset in $G$ such that $F_n \subset F_{n+1}$, $\bigcup_{n\in \NN} F_n = G$ and
\[
\frac{\sharp \big((F_n\cdot g) \Delta F_n\big)}{ \sharp F_n} \to 0 \mbox{~for~each~} g \in G.
\]
It follows that for any $\varepsilon>0$, there exists $m\in \NN$ such that $\sum_{x\in F_m} f(x) \geq \|f\|_1 - \varepsilon$. For the finite subset $F_m$, the F{\o}lner condition above provides an $n\in \NN$ such that
\[
\frac{\sharp (F_n \setminus F_n g)}{ \sharp F_n} \leq \frac{\varepsilon}{\sharp F_m}
\]
for each $g\in F_m$. Set $F_{n,m}:=\{x\in G: F_m \subseteq F_n^{-1} x\} = \bigcap_{g\in F_m} F_n  g$. Then we have
\[
\frac{\sharp (F_n \setminus F_{n,m})}{\sharp F_n} = \frac{\sharp \left( \bigcup_{g\in F_m} (F_n \setminus F_n g) \right)}{\sharp F_n}  \leq \frac{\sum_{g \in F_m} \sharp (F_n \setminus F_n g)}{\sharp F_n} \leq \varepsilon,
\]
which implies that
\[
\frac{\sharp F_{n,m}}{\sharp F_n} \geq \frac{\sharp (F_{n,m}\cap F_n)}{\sharp F_n} = 1- \frac{\sharp (F_n \setminus F_{n,m})}{\sharp F_n} \geq 1- \varepsilon.
\]
Now we have:
\begin{align*}
\|T_f\| &\geq \left\|\frac{1}{(\sharp F_n)^{1/2}}\chi_{F_n}\ast f\right\|_2^2 = \frac{1}{\sharp F_n}\sum_{x \in G} \left| \sum_{y\in G} \chi_{F_n}(xy^{-1}) f(y) \right|^2 = \frac{1}{\sharp F_n}\sum_{x \in G} \left| \sum_{y\in F_n^{-1} x} f(y) \right|^2\\
& \geq \frac{1}{\sharp F_n}\sum_{x \in F_{n,m}} \left| \sum_{y\in F_n^{-1} x} f(y) \right|^2 \geq \frac{\sharp F_{n,m}}{\sharp F_n} \cdot (\|f\|_1 - \varepsilon)^2 \geq (1-\varepsilon) \cdot (\|f\|_1 - \varepsilon)^2.
\end{align*}
Letting $\varepsilon \to 0$, we conclude that the operator norm of $T_f$ coincides with $\|f\|_1^2$.
\end{proof}

Thanks to Proposition \ref{prop:op norm and l1 norm}, we introduce the following notion:

\begin{defn}\label{defn:associated ops}
For a finitely generated group $G$ and a function $f \in \ell^1(G)$ with range in $[0,1]$, the \emph{associated operator of $f$} is defined to be $T_f:=T_{K_f}$, which is a bounded linear operator on $\ell^2(G)$ with operator norm $\|f\|_1^2$.
\end{defn}

Consequently, combining Theorem \ref{thm:main result general case} with Lemma \ref{lem:thm 4.4.4 in NY12} and Proposition \ref{prop:op norm and l1 norm}, we obtain the following:

\begin{thm}\label{thm:main result amenability}
Let $\{G_i\}_{i\in I}$ be a family of finitely generated amenable groups with finite generating set $S_i$ and the induced word length metric for each $i$. Consider the following conditions:
\begin{enumerate}
 \item For any $\varepsilon>0$, there exists $M>0$ and a family of functions $\{f_i\}_{i\in I}$ such that each $f_i \in \ell^2(G_i)_{1,+} \cap \ell^1(G_i)$ satisfies $\|f_i-s\cdot f_i\|_2\leq \varepsilon$ for $s\in S_i$ and $\|f_i\|_1 \leq M$.
 \item For any $\varepsilon>0$, there exists $M'>0$ and a family of functions $\{f_i\}_{i\in I}$ such that each $f_i \in \ell^2(G_i)_{1,+} \cap \ell^1(G_i)$ satisfies $\|f_i-s\cdot f_i\|_2\leq \varepsilon$ for $s\in S_i$ and the operator norm of the associated operator $T_{f_i}$ in Definition \ref{defn:associated ops} is bounded by $M'$.
 \item For any $\varepsilon>0$, there exists $N>0$ and a family of functions $\{f_i\}_{i\in I}$ such that each $f_i \in \ell^2(G_i)_{1,+}$ satisfies $\|f_i-s\cdot f_i\|_2\leq \varepsilon$ for $s\in S_i$ and $\sharp \supp(f_i) \leq N$.
 \item The family $\{(G_i,S_i)\}_{i\in I}$ is uniformly amenable.
 %, \emph{i.e.}, for any $\varepsilon>0$ there exists $D>0$ and a family of functions $\{f_i\}_{i\in I}$ such that each $f_i \in \ell^2(G_i)_{1,+}$ has $(\varepsilon,1)$-variation and $\supp(f_i) \subseteq B(1_{G_i}, D)$.
 \item The family $\{(G_i,S_i)\}_{i\in I}$ has uniform Property A.
 %, \emph{i.e.}, for any $\varepsilon>0$ there exists $D>0$ and a family of maps $\{\xi_i: G_i \to \ell^2(G_i)_{1,+}\}$ such that each $\xi_i$ has $(\varepsilon,1)$-variation and $\supp(\xi_i(x)) \subseteq B(x, D)$.
\end{enumerate}
Then (1) $\Leftrightarrow$ (2) $\Rightarrow$ (3) $\Rightarrow$ (4) $\Leftrightarrow$ (5). Additionally, if there exists $L>0$ such that $\sharp S_i \leq L$ for each $i \in I$, then all the conditions above are equivalent.
\end{thm}

In the case of a sequence of finite groups, Theorem \ref{thm:main result amenability} can be simplified as follows:

\begin{cor}\label{cor:main result amenability}
Let $\{G_n\}_{n\in \NN}$ be a sequence of finite groups with finite generating set $S_n$ and the induced word length metric for each $n \in \NN$, and $X$ be their coarse disjoint union. If there exists $L>0$ such that $\sharp S_n \leq L$ for each $n\in \NN$, then the following are equivalent:
\begin{enumerate}
 \item For any $\varepsilon>0$, there exists $M>0$ and a family of functions $\{f_n\}_{n\in \NN}$ such that each $f_n\in \ell^2(G_n)_{1,+}$ satisfies $\|f_n-s\cdot f_n\|_2\leq \varepsilon$ for $s\in S_n$ and $\|f_n\|_1 \leq M$.
 \item For any $\varepsilon>0$, there exists $M'>0$ and a family of functions $\{f_n\}_{n\in \NN}$ such that each $f_n \in \ell^2(G_n)_{1,+}$ satisfies $\|f_n-s\cdot f_n\|_2\leq \varepsilon$ for $s\in S_n$ and the operator norm of the associated operator $T_{f_n}$ is bounded by $M'$.
 \item For any $\varepsilon>0$, there exists $N>0$ and a family of functions $\{f_n\}_{n\in \NN}$ such that each $f_n \in \ell^2(G_n)_{1,+}$ satisfies $\|f_n-s\cdot f_n\|_2\leq \varepsilon$ for $s\in S_n$ and $\sharp \supp(f_n) \leq N$.
 \item The sequence $\{(G_n,S_n)\}_{n\in \NN}$ is uniformly amenable.
 \item The space $X$ has Property A.
\end{enumerate}
\end{cor}

\section{Questions}\label{sec:ques}

Recall that our main results (Theorem \ref{thm:main result general case} and \ref{thm:main result amenability}) indicate that the uniform Hulanicki-Reiter condition is indeed equivalent to some formally weaker versions. One might ask whether the same phenomenon happens for the Higson-Roe condition. More precisely:

\begin{question}\label{ques: hig-roe}
Let $\{G_i\}_{i\in I}$ be a family of finitely generated groups with finite generating set $S_i$ for each $i$. Assume that for any $\varepsilon>0$, there exist $M>0$ and a family of maps $\{\xi_i: G_i \to \ell^2(G_i)_{1,+}\cap \ell^1(G_i)\}_{i\in I}$ such that each $\xi_i$ has $(\varepsilon,1)$-variation and $\|\xi_i(x)\|_1 \leq M$ for $x\in G_i$. Does the family $\{(G_i,S_i)\}_{i\in I}$ have uniformly Property A?
\end{question}

It is easy to construct a counterexample and provide a negative answer:

\begin{ex}\label{eg:ceg to hig-roe}
Let $\mathbb{F}_2$ be the free group generated by $S=\{a,b\}$, and $\{1\}=N_0 \triangleleft N_1 \triangleleft N_2 \triangleleft \cdots$ be a sequence of normal subgroups in $\mathbb{F}_2$ with finite index and trivial intersection. Take $G_n:=\mathbb{F}_2/N_n$ and $S_n=\{aN_n, bN_n\}$ for each $n\in \NN$. Since $\mathbb{F}_2$ is non-amenable, the sequence $\{(G_n, S_n)\}_n$ does not have uniform Property A. However, taking $\xi_n(g):=\delta_{N_n} \in \ell^2(G_n) \cap \ell^1(G_n)$ for each $n \in \NN$ and $g\in G_n$, we obtain that each $\xi_n$ has $(0,1)$-variation and $\|\xi_n(g)\|_1=1$ for any $g\in G_n$. Hence we obtain a negative answer to Question \ref{ques: hig-roe}.
\end{ex}

In fact, Question \ref{ques: hig-roe} has a negative answer even for finite $I$. One can just take a group without Property A and choose a function $\xi$ as in Example \ref{eg:ceg to hig-roe}.

We realise that the obstruction to generalise the Higson-Roe condition is due to the lack of equivariance for the Higson-Roe functions, which allows too much flexibility for the choice. Hence in order to obtain Property A, we need some extra requirement on their supports. Currently we do not have any appropriate candidate, therefore we would like to pose the following question:

\begin{question}
Let $G$ be a finitely generated group. Is it possible to weaken the requirement on the support of the Higson-Roe functions as in Theorem \ref{thm:main result general case} or \ref{thm:main result amenability} while still guaranteeing that $G$ has Property A?
\end{question}

\bibliographystyle{plain}
\bibliography{bib_Property_A}

\end{document}